\documentclass[a4paper, 11pt]{article}

\usepackage{amsmath}
\usepackage{amsfonts}
\usepackage{amssymb}
\usepackage[english]{babel}
\usepackage{graphicx}
\usepackage{amsthm}
\usepackage{mathrsfs}
\usepackage{pgf,tikz}
\usepackage{ulem}
\usepackage{amstext} 
\usepackage{array}   
\newcolumntype{C}{>{$}c<{$}} 
\definecolor{uququq}{rgb}{0.25,0.25,0.25}
\newtheorem{thm}{Theorem}[section]
\newtheorem{cor}[thm]{Corollary}
\newtheorem{lem}[thm]{Lemma}

\usetikzlibrary{arrows}
\theoremstyle{definition}

\theoremstyle{definition}

\theoremstyle{definition}
\newtheorem{ex}[thm]{Example}

\newcommand{\Z}{\mathbb{Z}}

\newcommand{\F}{\mathbb{F}}

\newcommand{\comment}[1]{}
\numberwithin{equation}{section}

\begin{document}
\date{}
\title{
Additivity of symmetric and subspace designs}
\author{Marco Buratti \thanks{Dipartimento di Scienze di Base e Applicate per l’Ingegneria (S.B.A.I.), Sapienza
Universit\`a di Roma, Via Antonio Scarpa, 10, Italy, email: marco.buratti@uniroma1.it}\\
\\ 
Anamari Naki\'c \thanks{Faculty of Electrical Engineering and Computing,
University of Zagreb, Croatia, email: anamari.nakic@fer.hr}}
\date{\today}
\maketitle
\begin{abstract}
A $2$-$(v,k,\lambda)$ design is additive (or strongly additive) if it is possible to embed it in a suitable abelian group $G$ in such a way that 
its block set is contained in (or coincides with) the set of all the zero-sum $k$-subsets of $G$. 
Explicit results on the additivity or strong additivity of symmetric designs and subspace 2-designs are presented.
In particular, the strong additivity of PG$_d(n,q)$, which was known to be additive only for $q=2$ or $d=n-1$, is always established.



\end{abstract}



\section{Introduction}
We assume familiarity with the very basic notions of design theory and finite geometry. 
For the relevant background, we refer to \cite{BJL} and to some
chapters \cite{IT,DPS,Storme} of the Handbook of Combinatorial Designs.

A 
design ${\mathscr D}=(V,{\mathscr B})$ is {\it additive} under an abelian group $G$, or briefly $G$-additive,
if there exists an injective map $f$ from $V$ to $G$ such that $f(B)$ is zero-sum for every $B\in{\mathscr B}$.
Such a map $f$ will be called an {\it embedding} of $\mathscr D$ in $G$.
In the very special case that $f$ maps $\mathscr B$ precisely onto the set of all zero-sum $k$-subsets of $G$,
the design $\mathscr D$ is said to be {\it strongly $G$-additive} and the map $f$ is said to be a {\it strong embedding} of $\mathscr D$ in $G$.
In order to underline that an embedding $f$ is not strong we say that $f$ is a {\it smooth} embedding.



By simply saying that ${\mathscr D}$ is additive or strongly additive one means that ${\mathscr D}$
is $G$-additive or strongly $G$-additive for at least one group $G$.
In general, to determine whether a design $\mathscr D$ is additive seems to be difficult. 
To establish if it is strongly additive is even more difficult.
In the affirmative cases, it would be interesting to determine the smallest $G$'s
in which $\mathscr D$ can be smoothly or strongly embeddable, respectively.
Obviously, the smallest possible order for $G$ is the number of points.
This is why we say that a design $\mathscr D$ with $v$ points is {\it strictly} or {\it almost strictly} $G$-additive
when it is additive under a group $G$ of order $v$ or $v+1$, respectively.

In a recent paper \cite{Pavone} it is shown that there are designs which are additive but not strongly additive. 

The theory of additive designs was recently introduced by Caggegi, Falcone and Pavone \cite{CFP}. Other papers
from these authors on the topic or related topics are \cite{C,CF,CFP2,FP}. 
Besides their intrinsic beauty, additive designs have interesting connections with several branches of
discrete mathematics such as coding theory and additive combinatorics. 
We also note that the usage of zero-sum blocks in the construction of designs is frequent (see, e.g., \cite{BCHW, BH, EW, K, WW}),
and that there are combinatorial designs different from the classic ones
which can be considered {\it additive} as, for instance, the so-called {\it Heffter arrays} \cite{PD}.

In their seminal paper Caggegi, Falcone and Pavone proved that 
every $(v,k,\lambda)$ symmetric design is strongly additive but they were able to indicate a concrete group where it is strongly 
embeddable, that is $\Z_{k-\lambda}^{(v-1)/2}$, only in the hypothesis that the order $k-\lambda$ is a prime not dividing $k$.
Here we prove that \underline{every} symmetric $(v,k,\lambda)$ design is strongly additive
under the (huge) group $\Z_{k-\lambda}^v$ without any conditions on the parameters. 
We also prove that a cyclic symmetric design is smoothly additive under a relatively small group  
if suitable arithmetic conditions are met. 

In order to explain our results on the ``geometrical side", we need to give some notation and terminology.
Given a prime power $q$, we denote by $\F_q$, EA$(q)$ and $\F_q^*$ the finite field of order $q$,
its additive group (that is the elementary abelian group of order $q$) and its multiplicative group,
respectively. Also, AG$(n,q)$ and PG$(n,q)$ will denote, respectively, the $n$-dimensional affine and projective geometries
over $\F_q$. 
The classical designs of points and $d$-dimensional 
subspaces of AG$(n,q)$ and PG$(n,q)$ will be denoted by AG$_d(n,q)$ and PG$_d(n,q)$, respectively. 
A {\it $2$-$(n,k,\lambda)_q$ {\it subspace design} -- also called a $2$-$(n,k,\lambda)$ design over $\F_q$} or 
a {\it $q$-analog of a $2$-$(n,k,\lambda)$ design} -- is a classic 2-design of parameters $({q^n-1\over q-1},{q^k-1\over q-1},\lambda)$ design whose points are
those of PG$(n-1,q)$ and whose blocks are suitable $(k-1)$-dimensional subspaces of PG$(n-1,q)$.
In particular, PG$_d(n,q)$ is a 2-$(n+1,d+1,\lambda)_q$ design where $\lambda$ is the Gaussian coefficient ${n-1 \brack d-1}_q$;
we may call it the {\it complete} $n$-dimensional $d$-subspace design.
For general background on subspace designs we refer to \cite{BKW}.

Note 
that any coset of a subgroup of EA$(q^n)$ of order $q^d>2$ is zero-sum.
Hence the designs which are most obviously additive are AG$_d(n,q)$ with $(d,q)\neq(1,2)$ and the designs 
 over $\F_2$.
Indeed the blocks of AG$_d(n,q)$ are all the cosets of all the subgroups of $EA({q^n})$ of order $q^d$, and the blocks of
a 2-$(v,k,\lambda)_2$ design are suitable subgroups of $EA(2^v)$ of order $2^k$ deprived of the identity element.
Thus, besides the additivity of PG$_d(n,2)$, we have the existence of  an additive 2-$(v,7,7)$ design for every odd $v$ 
in view of the main results in \cite{BN1,Thomas}, and of an additive 2-(8191,7,1) design, that is the well-celebrated 
2-$(13,3,1)_2$ design found in $\cite{BEOVW}$ (see also Section 6.1 in \cite{BNW}). 
It is also evident the strict additivity of every 2-$(q^n,kq,\lambda)$ design whose blocks are union of $k$ parallel lines of AG$(n,q)$
(see \cite{BN3,CF,N} for some examples).
In all these cases the embedding map is simply the identity.

For $q>2$ the additivity of PG$_d(n,q)$ and, more generally, of a 2-design over $\F_q$ is not immediate.
One can deduce it only for PG$_{n-1}(n,q)$ since this design -- that of points and hyperplanes of PG$(n,q)$ --
is symmetric. 
In this paper we will prove that PG$_d(n,q)$ is always strongly additive under a huge group
and smoothly additive under the small EA$(q^{n+1})$. More generally, all subspace 2-designs 
are additive. Unfortunately, their strong additivity remains in doubt.


Our four main results, one for each of the subsequent sections, can be summarized as follows.
\begin{itemize}
\item[(1)] Every symmetric $(v,k,\lambda)$ design is  strongly $\Z_{k-\lambda}^v$-additive.
\item[(2)] A cyclic symmetric $(v,k,\lambda)$ design is smoothly $\Z_p^{ord_v(p)}$-additive for any prime $p$ dividing $k$ but not $v$.
\item[(3)] PG$_d(n,q)$ is strongly $\Z_{q^{d}}^{(q^{n+1}-1)/(q-1)}$-additive.
\item[(4)] Any $2$-$(n,k,\lambda)_q$ design is smoothly EA$(q^{n})$-additive.
\end{itemize}



Applying (3) with $d=1$ we get, in particular, that PG$_1(n,q)$ is an additive Steiner 2-design.
In \cite{BN2} we recently proved that for any $k$ which is neither singly even nor of 
the form $2^n3$ there are infinitely many -- unfortunately huge -- values of $v$ for which there exists a strictly additive 2-$(v,k,1)$ design.
So the big challenge is to determine additive 2-$(v,k,1)$ designs with $v$ ``reasonable" where
$k$ is neither a prime power nor a prime power plus one. Indeed for $k$ a prime power we have 
AG$_1(n,k)$ whereas for $k$ a prime power plus one we have PG$_1(n,k)$.

\smallskip
It is important to note that the additive designs by Caggegi et al. bear no known relation to the
“additive BIB designs” considered in \cite{Japan}, in spite of the misfortune of inadvertently similar terminology.

\section{Strong additivity of the symmetric designs}
 
 We recall that a $(v,k,\lambda)$ {\it symmetric} design is a $2$-$(v,k,\lambda)$ design with as many points as blocks. 
 As a consequence, the number of blocks through a point is equal to $k$ and any two distinct blocks share exactly $\lambda$ points. 
 The trivial necessary condition for its existence is that $\lambda(v-1)=k(k-1)$. Another necessary condition is
 given by the famous theorem of Bruck, Ryser and Chowla \cite{BJL,IT}.
 
In this section we get the strong additivity of any symmetric design already obtained by Caggegi et al. in \cite{CFP}.
Anyway, differently from them, we are always able to indicate a group under which the
strong additivity is realizable.

 We first need the following lemma.
 \begin{lem}\label{doublecounting}
 If $X$ is a $k$-subset of a $(v,k,\lambda)$ symmetric design ${\mathscr D}=(V,{\mathscr B})$, then there exists a block $B\in{\mathscr B}$
 intersecting $X$ in at least $\lambda+1$ points.
 \end{lem}
 \begin{proof}
Consider the set $\Phi(X):=\{(x,B) \ | \ x\in X; x\in B\in {\mathscr B}\}$ of all {\it flags} of $\mathscr D$ having the point in $X$.
 The number of pairs $(\overline x,B)$ belonging to $\Phi(X)$ with $\overline x$ fixed, is equal to the number of blocks of $\mathscr D$
 containing $\overline x$, that is $k$. Thus we have $|\Phi(X)|=k^2$. The number of pairs $(x,\overline B)$ belonging to $\Phi(X)$ 
 with $\overline B$ fixed, is clearly equal to $|\overline B\cap X|$. Thus we have $|\Phi(X)|=\sum_{B\in{\mathscr B}}|B\cap X|$.
 Comparing the obtained equalities we get
 $$k^2=\sum_{B\in{\mathscr B}}|B\cap X|$$
 Assume for contradiction that $|B\cap X|\leq\lambda$ for every $B\in{\mathscr B}$. In this case the above equality
 would give $k^2\leq \lambda v$. This, together with the trivial identity $\lambda(v-1)=k^2-k$, would imply
 that $\lambda\geq k$ which is absurd.
 \end{proof}
 
 \begin{thm}\label{strongsymmetric}
 A symmetric $(v,k,\lambda)$ design is strongly additive under $\Z_{k-\lambda}^v$.
 \end{thm}
 \begin{proof}
 Given a symmetric $(v,k,\lambda)$ design ${\mathscr D}=(V,{\mathscr B})$,
 take an ordering $\{x_1,\dots,x_v\}$ of $V$ and an ordering $\{\beta_1,\dots,\beta_v\}$ of $\mathscr B$.
Consider the $v\times v$ matrix $M=(m_{i,j})$  with entries in $\Z_{k-\lambda}$ defined by
$$m_{i,j}=\begin{cases}
0 & {\rm if} \ x_i\in\beta_j\cr
1 & {\rm if} \ x_i\not\in \beta_j
 \end{cases}$$
and consider the map
$$f: x_i \in V \longrightarrow M_i\in \Z_{k-\lambda}^v$$
where $M_i$ denotes the $i$-th row of $M$.
For any two distinct points $x_{i_1}$, $x_{i_2}$,
there is a block $\beta_j$ containing $x_{i_1}$ but not $x_{i_2}$ since the number of blocks 
through $x_{i_1}$, that is $k$, is strictly greater than the 
number of blocks containing both $x_{i_1}$ and $x_{i_2}$, that is $\lambda$.
It follows that the $j$-th component of $M_{i_1}$ is 0 whereas the $j$-th component of $M_{i_2}$ is 1
and then $f(x_{i_1})\neq f(x_{i_2})$. Thus $f$ is injective. By definition of strong additivity,
it is enough to prove that $\beta$ is a block of $\mathscr D$ if and only if $\beta$
is a $k$-subset of $V$ such that $f(\beta)$ is zero-sum.

Let $\beta=\{x_{i_1},x_{i_2},\dots,x_{i_k}\}$ be a fixed block of $\mathscr D$.
Given any block $\beta_j$ we have either $\beta=\beta_j$ or $|\beta_j \ \cap \ \beta|=\lambda$.
In the first case we have $m_{i_h,j}=0$ for $1\leq h\leq k$. In the second case
we have $m_{i_h,j}=1$ for exactly $k-\lambda$ values of $h\in\{1,\dots,k\}$, that are
the values of $h$ for which $x_{i_h}\in\beta\setminus\beta_j$.
In both cases we clearly have $m_{i_1,j}+m_{i_2,j}+ \ldots +m_{i_k,j}=0$.
This is true for any $j$, hence $M_{i_1}+M_{i_2}+\dots+M_{i_k}=0$, i.e., $f(\beta)$ is zero-sum.

Now assume that $\beta=\{x_{i_1},\dots,x_{i_k}\}$ is a $k$-subset of $V$
such that $f(\beta)$ is zero-sum. Thus we have $M_{i_1}+M_{i_2}+\dots+M_{i_k}=0$ and then
\begin{equation}\label{proiezione2}
m_{i_1,j}+m_{i_2,j}+\ldots+m_{i_k,j}=0\quad\mbox{in $\Z_{k-\lambda}$ for $1\leq j\leq v$}
\end{equation}
By Lemma \ref{doublecounting} there is a block $\beta_j$ intersecting $\beta$ in at least $\lambda+1$ points.
Thus, up to a reordering of the indices, we can assume that $\{x_{i_1},x_{i_2},\dots,x_{i_{\lambda+1}}\}$ is contained in $\beta_j$
so that we have $m_{i_h,j}=0$ for $1\leq h\leq \lambda+1$.
It follows that $m_{i_h,j}=0$ also for $\lambda+2\leq h\leq k$ 
in view of (\ref{proiezione2}) and the fact that the entries of $M$ are only $0$'s and $1$'s. Hence every point of $\beta$ is in $\beta_j$.
This means that $\beta=\beta_j$ and the assertion follows.
 \end{proof}

 \section{Smooth additivity of cyclic symmetric designs}
 
 An incidence structure $(V,{\mathscr B})$  is said to be {\it cyclic} if there exists a cyclic permutation on $V$ leaving $\mathscr B$ invariant.
It is very well-known that, up to isomorphism, every cyclic symmetric $(v,k,\lambda)$ design is 
of the form $(\Z_v,\{D+i \ | \ 0\leq i\leq v-1\})$ where $D$ is a cyclic $(v,k,\lambda)$ {\it difference set}. 
This means that $D$ is a $k$-subset of $\Z_v$ such that its {\it list of differences}
$\Delta D=\{d-d' \ | \ d,d'\in D;d\neq d'\}$ is exactly $\lambda$ times $\Z_v\setminus\{0\}$. 

From the previous section we already know that every symmetric $(v,k,\lambda)$ design $\mathscr D$ is additive under $\Z_{k-\lambda}^v$.
It is clear that the order of this group is in general huge. Here we prove that if $\mathscr D$ is cyclic and the radical of
$k-\lambda$ does not divide $v$, then $\mathscr D$ is additive under a group of much smaller order.

In the following, assuming that $\gcd(u,v)=1$, we denote by $ord_v(u)$ the multiplicative order of $u$ modulo $v$, i.e., the order of 
$u$ in the group of units of $\Z_v$.

\begin{thm}\label{wow}
Let $\mathscr D$ be a cyclic symmetric $(v,k,\lambda)$ design and let $p$ be a prime dividing $k-\lambda$ but not $v$. 
Then $\mathscr D$ is $EA({p^t})$-additive with $t=ord_v(p)$.
\end{thm}
\begin{proof}
By definition of $t$, we have $p^t\equiv1$ (mod $v$) so that $v$ is a divisor of the order $p^t-1$ of $\F_{p^t}^*$.
Let $g$ be a generator of the subgroup of $\F_{p^t}^*$ of order $v$
and consider the injective maps $f_1$ and $f_{-1}$ defined as follows:
$$f_1: x\in \Z_v \longrightarrow g^{x} \in \F_{p^t},\quad\quad\quad f_{-1}: x\in \Z_v \longrightarrow g^{-x} \in \F_{p^t}.$$
The assertion will be proved if we show that there is at least one $i\in\{1,-1\}$ such that
\begin{equation}\label{sufficient}
\sum_{b\in B}f_i(b)=0\quad\quad\mbox{for every block $B$ of $\mathscr D$}
\end{equation}
By assumption, the blocks of $\mathscr D$ are all the translates of a $(v,k,\lambda)$ difference set $D$. 
Consider the two sums  $$\displaystyle\sigma_1:=\sum_{d\in D}f_1(d)=\sum_{d\in D}g^d,\quad\quad\quad \displaystyle\sigma_{-1}:=\sum_{d\in D}f_{-1}(d)=\sum_{d\in D}g^{-d}$$
and let us calculate their product:
$$\sigma_1\cdot\sigma_{-1}=\sum_{d\in D}g^{d-d}+\sum_{d\in D} \ \sum_{d'\in D\setminus\{d\}}g^{d-d'}=$$

$$=|D|+\sum_{\delta\in\Delta D}g^\delta=k+\lambda\cdot(\sum_{i=1}^{v-1}g^i)=$$

$$=(k-\lambda)+\lambda\cdot(\sum_{i=0}^{v-1}g^i)=k-\lambda=0.$$
The penultimate equality holds since the order of $g$ in $\F_{p^t}^*$ is $v$ so that $g^v=1$ and then $\displaystyle\sum_{i=0}^{v-1}g^i={g^v-1\over g-1}=0$.
The last equality holds since by assumption $k-\lambda$ is divisible by $p$ and hence $k-\lambda$ is null in $\F_{p^t}$.

From the obtained equality $\sigma_1\cdot\sigma_{-1}=0$ we infer that one of the two sums $\displaystyle\sigma_1$ and $\displaystyle\sigma_{-1}$ is equal to 0,
i.e., there is one element $i\in\{1,-1\}$ such that $\sigma_i=0$. 
Take this element $i$, take any block $B=D+j$ of $\mathscr D$, and note that the sum of the elements
of $f_i(B)$ is $\displaystyle\sum_{d\in D}g^{i(d+j)}=g^{ij}\cdot\displaystyle\sum_{d\in D}g^{id}=g^{ij}\cdot\sigma_i=0$. Hence (\ref{sufficient}) holds and the assertion follows.
\end{proof}

The above theorem gives many concrete results when $k-\lambda$ is coprime with $v$.
Anyway we may have applications of the theorem also for $\gcd(k-\lambda,v)\neq1$.
For instance, according to \cite{Gordon} the existence of cyclic difference sets of parameters $(465,145,45)$ 
and $(910,405,180)$ is still open.
If they exist, using Theorem \ref{wow} we can say that their associated symmetric designs would be $\Z_2^{20}$-additive and
$\Z_{3}^{12}$-additive,  respectively.
Indeed in the first case 2 is a divisor of $k-\lambda=100$ which does not divide $v=465$
and we have $ord_{465}(2)=20$. In the second case 3 is a divisor of $k-\lambda=225$ which does not divide $v=910$
and we have $ord_{910}(3)=12$. 

We recall that the incidence structure points-hyperplanes of PG$(n,q)$ is a cyclic symmetric design
generated by the so-called {\it Singer} $({q^{n+1}-1\over q-1},{q^{n}-1\over q-1},{q^{n-1}-1\over q-1})$ difference set.
Here $k-\lambda={q^{n}-1\over q-1}-{q^{n-1}-1\over q-1}=q^{n-1}$ is obviously coprime with $v={q^{n+1}-1\over q-1}$.
Assume that $q=p^\alpha$ with $p$ prime and
note that we have $$p^{\alpha(n+1)}=q^{n+1}=(q-1)v+1\equiv1 \quad (mod \ v)$$ 
This means that $ord_v(p)=\alpha(n+1)$. Thus, applying Theorem \ref{wow}
we get the following.

\begin{cor}
The point-hyperplane design of PG$(n,q)$ is $EA({q^{n+1}})$-additive.
In particular, the desarguesian projective plane of order $q$ is $EA({q^3})$-additive.
\end{cor}

\begin{ex}
Consider the Singer $(13,4,1)$ difference set $D=\{0,1,3,9\}$ generating PG$(2,3)$, the projective plane of order 3. 
Let $r$ be a root of the primitive polynomial $x^3+2x^2+1$ over $\F_3$. Taking $r$ as primitive element of $\F_{3^3}$,
a generator of the subgroup of $\F_{3^3}^*$ of order 13 is $g=r^2$. Let us calculate the two sums $\sigma_1$ and $\sigma_{-1}$.
$$\sigma_1=g^0+g^1+g^3+g^9=r^0+r^2+r^6+r^{18}=$$
$$=(0,0,1)+(1,0,0)+(2,2,0)+(0,1,1)=(0,0,2);$$

$$\sigma_{-1}=g^0+g^{-1}+g^{-3}+g^{-9}=r^0+r^{-2}+r^{-6}+r^{-18}=$$
$$=(0,0,1)+(0,2,1)+(2,0,2)+(1,1,2)=(0,0,0).$$

Thus a smooth embedding of PG$(2,3)$ in $\Z_3^3$ is given by 
the map $$f^{-1}: x\in \Z_{13} \longrightarrow g^{-x}\in \F_{3^3}$$
In other words
PG$(2,3)$ can be seen as the design $(V,\mathscr{B})$ where 
$$V=\{001,100,122,220,112,121,120,020,201,011,202,111,021\}$$
and where $\mathscr{B}$ consists of the following zero-sum blocks
$$\{001,021,202,112\},\quad\{021,111,011,220\},\quad\{111,202,201,122\},$$
$$\{202,011,020,100\},\quad\{011,201,120,001\},\quad\{201,020,121,021\},$$
$$\{020,120,112,111\},\quad\{120,121,220,202\},\quad\{121,112,122,011\},$$
$$\{112,220,100,201\},\quad\{220,122,001,020\},\quad\{122,100,021,120\},$$
$$\{100,001,111,121\}.$$
\end{ex}

Another important class of symmetric designs is that of Paley. For any given prime 
$v=4\lambda+3$, the set $\Z_v^\Box$ of non-zero squares of $\F_v$ is the
so-called Paley $(4\lambda+3,2\lambda+1,\lambda)$ difference set. 
Let Paley$(v)$ be its associated symmetric design. 
Here $k-\lambda=(2\lambda+1)-\lambda=\lambda+1$ is clearly coprime with $v=4\lambda+3$. 
Hence, applying Theorem \ref{wow} we get the following.

\begin{cor}
Let $v=4\lambda+3$ be a prime, let $p$ be any prime divisor of $\lambda+1$, and let $t=ord_{v}(p)$. 
Then Paley$(v)$ is additive under $\Z_p^{t}$.
\end{cor}

Note, in particular, that if $v=4\lambda+3$ is a Mersenne prime, say $v=2^t-1$, then
2 divides $k-\lambda$ but not $v$ and $ord_v(2)=t$.
Thus Theorem \ref{wow} allows to state the following.

\begin{cor}
If $v=2^t-1$ is a Mersenne prime, then Paley$(v)$ is almost strictly $\Z_{2}^t$-additive.
\end{cor}

 \section{Strong additivity of PG$_d(n,q)$}
 
Throughout this section, given a prime power $q$ and a positive integer $n$, we denote by $[n]_q$
the number of points of PG$(n-1,q)$, hence
$$[n]_q={q^{n}-1\over q-1}$$
\begin{thm}
PG$_d(n,q)$ is strongly additive under $\Z_{q^d}^{[n+1]_q}$.
\end{thm}
\begin{proof}
Set $v=[n+1]_q$ and $k=[d+1]_q$.
Let ${\mathscr  P}=\{x_1,\dots,x_v\}$ be an ordering of the points of  PG$(n,q)$,
and let ${\cal H}=\{\pi_1,\dots,\pi_v\}$ be an ordering of its hyperplanes.

Consider the $v\times v$ matrix $M=(m_{i,j})$ with entries in $\Z_{q^d}$ defined by
$$m_{i,j}=\begin{cases}
0 & {\rm if} \ x_i\in\pi_j\cr
1 & {\rm if} \ x_i\not\in \pi_j
 \end{cases}$$
and consider the map
$$f: x_i \in {\mathscr P} \longrightarrow M_i\in \Z_{q^d}^v$$
where $M_i$ denotes the $i$-th row of $M$.
Given any two distinct points $x_{i_1}$, $x_{i_2}$, we can take a hyperplane
$\pi_j$ containing  $x_{i_1}$ but not $x_{i_2}$ so that the $j$-th component of $f(x_{i_1})$ is 0 whereas 
the $j$-th component of $f(x_{i_2})$ is 1. Hence $f(x_{i_1})\neq f(x_{i_2})$ and $f$ is injective.
The assertion will be proved if we show that the following equivalence holds.
$$\mbox{$S$ is a $d$-subspace of PG$(n,q)$ \quad$\Longleftrightarrow$\quad $S\in{{\mathscr P}\choose k}$ and $f(S)$ is zero-sum.}$$
$(\Longrightarrow).$\quad Let $S=\{x_{i_1},\dots,x_{i_k}\}$ be a $d$-subsapce of PG$(n,q)$. 
For any hyperplane $\pi_j$ of $\cal{H}$
we have either $S\subset \pi_j$ or $|S \cap \pi_j|=[d]_q$. In the first case
we have $m_{{i_h},j}=0$ for $1\leq h\leq k$ and then $m_{i_1,j}+m_{i_2,j}+\dots+m_{i_k,j}=0$.
In the second case, up to a reordering of the indices we can assume that $S\cap\pi_j=\{x_{i_1},\dots,x_{i_{[d]_q}}\}$ so that we have
$$m_{i_h,j}=\begin{cases}
0 & {\rm for} \ 1\leq h\leq [d]_q\medskip\cr
1 & {\rm for} \ [d]_q+1\leq h\leq k
 \end{cases}$$
 We get again that $m_{i_1,j}+m_{i_2,j}+\dots+m_{i_k,j}$ is null since it is the sum of $k-[d]_q=[d+1]_q-[d]_q=q^d$ ones
 and the sum has to be performed in $\Z_{q^d}$.
In view of the arbitrariness of $j$,
it follows that $M_{i_1}+M_{i_2}+\dots+M_{i_k}=0$, i.e., $f(S)$ is zero-sum.

$(\Longleftarrow).$\quad Assume that $S=\{x_{i_1},x_{i_2},\dots,x_{i_k}\}$ is a $k$-subset of $\mathscr P$
such that $f(S)$ is zero-sum. Thus we have $M_{i_1}+M_{i_2}+\dots+M_{i_k}=0$ and then
\begin{equation}\label{proiezione}
m_{i_1,j}+m_{i_2,j}+\dots+m_{i_k,j}=0\quad\mbox{in $\Z_{q^d}$ for  $1\leq j\leq v$}
\end{equation}

Consider the set $\Phi$ of all pairs $(x_{i_h},\pi_j)$ with $1\leq h\leq k$ and $\pi_j\ni x_{i_h}$. 
By a very similar counting argument as that used in the proof of Lemma \ref{doublecounting} we get
\begin{equation}\label{doublecounting2}
 k\cdot[n]_q=\sum_{j=1}^v|\pi_j\cap S|
 \end{equation}
 Indeed the number of pairs $(x_{i_h},\pi_j)\in \Phi$ with $h\in\{1,\dots,k\}$ fixed, is equal to the number of hyperplanes
 through $x_{i_h}$, that is $[n]_q$. Thus $\Phi$ has size equal to the left-hand side of (\ref{doublecounting2}). Also, the pairs $(x_{i_h},\pi_j)\in\Phi$ 
 with $j$ fixed, is equal to $|\pi_j\cap S|$ so that $\Phi$ has size equal to the right-hand side of (\ref{doublecounting2}). 
 
Let $J$ be the set of $j$'s for which $\pi_j$ intersects $S$ in more than $[d]_q$ points and set $\overline J=\{1,2,\dots,v\}\setminus J$.

Take any $j$ in $J$. By definition of $J$, up to a reordering of the indices $i_1$, $i_2$, \dots, $i_k$, we can assume that $\{x_{i_1},x_{i_2},\dots,x_{i_{[d]_q+1}}\}$ 
is contained in $\pi_j \ \cap \ S$ so that we have $m_{i_h,j}=0$ for $1\leq h\leq [d]_q+1$.
It follows that $m_{i_1,j}+m_{i_2,j}+\dots+m_{i_k,j}$ is the sum of a number of $1$s which is at most equal to $k-([d]_q+1)=[d+1]_q-[d]_q-1=q^d-1$. 
Considering (\ref{proiezione2}), this number is necessarily zero, i.e., $m_{i_h,j}=0$ for $1\leq h\leq k$. 
We conclude that every point of $S$ is in $\pi_j$, i.e., $S\subset \pi_j$.

Thus the contribute of each $j\in J$ to the sum in (\ref{doublecounting2}) 
is exactly equal to $|S|=k$
whereas the contribute of each $j\in \overline J$ is at most equal to $[d]_q$ by definition of $\overline J$. 
Thus we can write
$$\sum_{j=1}^v|\pi_j\cap S|\leq k|J|+(v-|J|)[d]_q$$
and then, by (\ref{doublecounting2}),
$$k\cdot[n]_q \leq k|J|+(v-|J|)[d]_q$$
which is equivalent to $|J|\geq{k\cdot[n]_q-v[d]_q\over k-[d]_q}$. Now we have
$${k\cdot[n]_q-v[d]_q\over k-[d]_q}=$$
$$=\biggl{(}{q^{d+1}-1\over q-1}\cdot{q^{n}-1\over q-1}-{q^{n+1}-1\over q-1}\cdot{q^{d}-1\over q-1}\biggl{)}\cdot{1\over{q^{d+1}-1\over q-1}-{q^{d}-1\over q-1}}=$$
$$={-q^{d+1}-q^n+q^{n+1}+q^d\over (q-1)^2}\cdot{q-1\over q^{d+1}-q^d}=$$
$$={(q^n-q^d)(q-1)\over(q-1)^2}\cdot{q-1\over q^d(q-1)}={q^{n-d}-1\over q-1}=[n-d]_q$$
so that we can write
\begin{equation}\label{J>=}
|J|\geq [n-d]_q
\end{equation}
Let $S'$ be the subspace of PG$(n,q)$ which is the intersection of all the hyperplanes $\pi_j$ with $j\in J$
and let $d'$ be its dimension. We have already seen that $S$ is contained in every $\pi_j$ with $j\in J$ so that $S$ is also contained in $S'$.
It follows that $[d+1]_q=|S|\leq|S'|=[d'+1]_q$ and then $d\leq d'$. 

Now let $\Pi(S')$ be the pencil of hyperplanes containing $S'$. 
By the principle of duality, $\Pi(S')$ has size equal to the number
of points belonging to a $(n-d'-1)$-subspace, that is $[n-d']_q$.
By definition of $S'$, we have $\{\pi_j \ | \ j\in J\}\subset \Pi(S')$ so that we have $|J|\leq [n-d']_q$.
This, together with (\ref{J>=}), gives $[n-d]_q\leq [n-d']_q$ and hence $d'\leq d$. 
Anyway we already noted that $d\leq d'$ so that $d=d'$ and then
$|S|=[d]_q=[d'_q]=|S'|$. Considering that $S$ is contained in $S'$ we conclude that $S=S'$, i.e., $S$
is a $d$-subspace of PG$(n,q)$ and the assertion follows.
 \end{proof}
 
Applying the above theorem with $d=n-1$ we get that the point-hyperplane design associated with PG$(n,q)$ 
 is strongly additive under $\Z_{q^{n-1}}^{[n+1]_q}$. Note that this is exactly the same result obtainable by 
 applying Theorem \ref{strongsymmetric}.
 
 \section{Additivity of subspace designs}

In the previous section we established the strong additivity of PG$_d(n,q)$ under a quite huge group.
Here we show that if we are content with the smooth additivity, it is enough to suitably
embed it in the elementary abelian group of order $q^{n+1}$. This is a special case of the following.

\begin{thm}
Every $2$-$(n,k,\lambda)_q$ design is EA$(q^{n})$-additive.
\end{thm}
\begin{proof}
Let us take, as it is standard, $\F_{q^{n}}^*/\F_q^*$ as point set of PG$(n-1,q)$.
Note that two points $\overline x$ and $\overline y$ coincide if and only if ${x\over y}\in\F_q^*$, which is equivalent to saying that $({x\over y})^{q-1}=1$
since $\F_q^*$ is clearly the group of $(q-1)$-th roots of unity of $\F_{q^{n}}$. Thus we have:
$$\overline x=\overline y \Longleftrightarrow x^{q-1}=y^{q-1}.$$ 
This means that the map
$$f: \overline x \in \F_{q^{n}}^*/\F_q^* \longrightarrow x^{q-1}\in \F_{q^{n}}$$
is well-defined and injective.

Recalling that every block of a $2$-$(n,k,\lambda)_q$ design is a subspace of PG$(n-1,q)$, 
the assertion will be proved if we show that we have
\begin{equation}\label{zerosum}
\displaystyle\sum_{\overline z \in S}f(\overline{z})=0\quad \mbox{for every subspace $S$ of PG$(n-1,q)$}
\end{equation}
Given any two distinct points $\overline x$ and $\overline y$ of a line $L$ we have
\begin{equation}\label{VP}
L\setminus\{\overline x\}=\{\overline{\lambda x+y} \ | \ \lambda\in\F_q\}
\end{equation}
Using the binomial Newton's formula we can write:
\begin{equation}\label{VP2}
\begin{array}{l}
\displaystyle\sum_{\lambda\in\F_q}(\lambda x+y)^{q-1}=\sum_{\lambda\in\F_q}\sum_{i=0}^{q-1}{q-1\choose i}\lambda^ix^iy^{q-1-i}=\medskip\\
=\displaystyle\sum_{i=0}^{q-1}{q-1\choose i}\mu_ix^iy^{q-1-i}\\
\end{array}
\end{equation}
where $\mu_i=\displaystyle\sum_{\lambda\in\F_q}\lambda^i$.
It is well-known (see, e.g.,  the lemma on page 5 of \cite{Serre}) that we have:
$$\sum_{\lambda\in\F_q}\lambda^i=\begin{cases}
0 & {\rm for} \ 0\leq i\leq q-2;\cr
-1 & {\rm for} \ i=q-1
 \end{cases}$$
 It follows that the only non-null addend in the last sum of (\ref{VP2}) is the one corresponding to $i=q-1$, that is $-x^{q-1}$. Then, considering (\ref{VP}),
 we have $\displaystyle\sum_{\overline z \in L}f(\overline{z})=0$. This proves that (\ref{zerosum}) holds for
 1-dimensional subspaces. Now take any $d$-dimensional subspace $S$ of PG$(n-1,q)$, take a point $\overline x$ of $S$ and 
 let $\mathscr{L}$ be the set of lines of $S$ through $\overline x$. 
Consider the double sum 
 $$\sigma=\sum_{L\in\mathscr{L}}\sum_{\overline z\in L}f(\overline z)$$
 Given that (\ref{zerosum}) has been proved for the lines, $\sum_{\overline z\in L}f(\overline z)$ is null for every $L\in\mathscr{L}$
and hence $\sigma$ is null as well. Also note that in the expansion of $\sigma$ the addend $f(\overline y)$ appears exactly once
for every $\overline y\in S\setminus\{\overline x\}$ whereas the addend $f(\overline x)$ appears exactly $|\mathscr{L}|$ times.
 Thus we can write
 $$\sigma=0=\sum_{\overline z\in S}f(\overline z)+(|\mathscr{L}|-1)f(\overline x)$$
 Now note that $\mathscr{L}$ has size ${q^d-1\over q-1}$ so that $|\mathscr{L}|-1=q\cdot{q^{d-1}-1\over q-1}$ which is zero in $\F_q$.
We conclude that $\sum_{\overline z\in S} f(\overline z)=0$, hence (\ref{zerosum}) is completely proved and the assertion follows.
 \end{proof}
 
As a consequence, recalling that PG$_d(n,q)$ is a 2-$(n+1,d+1,{n-1 \brack d-1}_q)_q$ design, we can state the following.
 
 \begin{cor}
PG$_d(n,q)$ is EA$(q^{n+1})$-additive.
\end{cor}

\begin{ex} Let us give an additive representation of PG$_1(3,3)$.
Let $g$ be a root of the primitive polynomial $x^4+x+2$ over $\mathbb{F}_{3}$. 
A standard presentation of PG$_1(3,3)$ is the following:

point set $V=\mathbb{F}_{3^4}^*/\mathbb{F}_{3}^*$; 

block set $\mathscr{B}=\{B_i\cdot g^j \ | \ 1\leq i\leq 3; 0\leq j\leq 39\} \ \cup \ \{B_4\cdot g^j \ | \ 0\leq j\leq 9\}$ 
where 
$$B_1=\{g^0,g^1,g^4,g^{13}\};\quad B_2=\{g^0,g^2,g^{17},g^{24}\};$$ 
$$B_3=\{g^0,g^5,g^{26},g^{34}\};\quad B_4=\{g^0,g^{10},g^{20},g^{30}\}.$$
The map 
$$f: g^i \in \mathbb{F}_{3^4}^*/\mathbb{F}_{3}^* \longrightarrow g^{2i}\in \F_{3^4}$$
turns our design $(V,{\mathscr B})$ into the isomorphic design $(f(V),f({\mathscr B}))$ where the point set $f(V)$
is the set of non-zero squares of $\F_{3^4}$ and where each block $f(B)$ is zero-sum. For instance we
have:
$$f(B_1)=\{g^0,g^2,g^8,g^{26}\}=\{(0,0,0,1),(0,1,0,0),(0,1,1,1),(0,1,2,1)\};$$
$$f(B_2)=\{g^0,g^4,g^{34},g^{48}\}=\{(0,0,0,1),(0,0,2,1),(0,1,2,2),(0,2,2,2)\};$$
$$f(B_3)=\{g^0,g^{10},g^{52},g^{68}\}=\{(0,0,0,1),(1,1,2,1),(1,0,0,2),(1,2,1,2)\};$$
$$f(B_4)=\{g^0,g^{20},g^{40},g^{60}\}=\{(0,0,0,1),(2,2,1,0),(0,0,0,2),(1,1,2,0)\}.$$

\end{ex}

Unfortunately, we are not able to answer the question on whether there are strongly additive subspace 2-designs which 
are not complete. We think that investigating this question 
is much worth of attention since, maybe, it could give some answers also on the very hard existence problem for non-complete designs over a finite field.

\section*{Acknowledgements}
This work has been performed under the auspices of the G.N.S.A.G.A. of the C.N.R. (National Research Council) of Italy.

The second author is supported in part by the Croatian Science Foundation
under the project 9752.

 \bibliographystyle{model1-num-names}

\end{document}